\documentclass[a4paper,draft]{amsproc}
\usepackage{amssymb}
\usepackage{amsmath}
\usepackage{amscd} %% Package for commutative diagrams
\DeclareMathOperator{\sech}{sech}
\DeclareMathOperator{\Isom}{Isom}

\theoremstyle{plain}
 \newtheorem*{thm}{\sc Main Construction}
 
 \newtheorem{lem}{Lemma}[section]
 
\theoremstyle{definition}

\theoremstyle{remark}
 \newtheorem{rem}{Remark}[section]
 \newtheorem{ques}{Question}[section]
 \newtheorem{conj}{Conjecture}[section]
 \numberwithin{equation}{section}

\renewcommand{\leq}{\leqslant}
\renewcommand{\geq}{\geqslant}

\title[Manifolds With Many Hyperbolic Planes]{Manifolds With Many Hyperbolic Planes}

\subjclass[2010]{}


\author[Samuel Lin and Benjamin Schmidt]{ Samuel Lin and Benjamin Schmidt}

\address{
Department of Mathematics \\ % \hfill (Received 00 00 2010)\\
Michigan State university   \\ %\hfill (Revised  00 00 2010)\\
East lansing, MI, 48823\\
}
\email{linsamue@msu.edu and schmidt@math.msu.edu}

\thanks{ } %% optional

\dedicatory{ }
\begin{document}

%{\begin{flushleft}\baselineskip9pt\scriptsize
%PUBLICATIONS DE L'INSTITUT MATH\'EMATIQUE\newline
%Nouvelle s\'erie, tome 91(105) (2012), od--do \hfill DOI:
%\end{flushleft}}
\vspace{18mm} \setcounter{page}{1} \thispagestyle{empty}

\begin{abstract}
We construct examples of complete Riemannian manifolds having the property that every geodesic lies in a totally geodesic hyperbolic plane.  Despite the abundance of totally geodesic hyperbolic planes, these examples are not locally homogenous.  

%The real hyperbolic metric, expressed in normal coordinates about a totally geodesic $\mathbb{H}^{n-1} \subset \mathbb{H}^n$, is a warped product of $\mathbb{H}^{n-1}$ over $\mathbb{R}$.  The (derivative of the) warping function satisfies the Riccati equation; the examples are obtained by perturbing the warping function in the solution set to this differential equation.  
\end{abstract}

\maketitle

\section{Introduction}

Hyperbolic $n$-space is the space $\mathbb{H}^n=\{(x_i)\in \mathbb{R}^n \vert\, x_n>0\}$ with the metric $h_n=x_n^{-2}(dx_1^2+\cdots +dx_n^2).$  Each tangent $2$-plane to $(\mathbb{H}^n,h_n)$ exponentiates to a totally-geodesic copy of $(\mathbb{H}^2,h_2)$. Similarly,   hyperbolic surfaces abound in the negatively curved locally symmetric manifolds.  More precisely, these spaces satisfy:\\

\noindent \textbf{(P)}: \textit{Each tangent vector to $M$ is contained in a tangent $2$-plane $\sigma$ that exponentiates to an immersed totally-geodesic hyperbolic surface.}\\

We construct examples of manifolds with (P) that are \textit{not} locally homogeneous.  For $n \geq 2$, let $M^n=\mathbb{R} \times \mathbb{\mathbb{H}}^{n-1}=\{(t,(x_i))\,\vert\, t \in \mathbb{R},\,(x_i)\in\mathbb{H}^{n-1}\}.$    

\begin{thm}\label{mainthm}
For each $(a,b) \in \mathbb{R} \times [-1,1]$, $M^n$ admits a complete Riemannian metric $h_{n,a,b}$ satisfying\begin{enumerate}
%\item For each $(a,b)$, $(M^2,h_{2,a,b})$ is isometric to $(\mathbb{H}^2,h_2)$.

\item For fixed $n \geq 2$, the metrics $h_{n,a,b}$ depend smoothly on $(a,b)$.

\item There is a monomorphism $\Isom(\mathbb{H}^{n-1},h_{n-1}) \rightarrow \Isom(M^n,h_{n,a,b}).$

\item Each tangent $2$-plane to $M^n$ containing the tangent vector $\frac{\partial}{\partial t}$ exponentiates to a totally-geodesic surface isometric to the hyperbolic plane.

\item The sectional curvature of a $2$-plane $\sigma$ at $p=(t,(x_i))$ making angle $\theta \in [0,\frac{\pi}{2}]$ with $\frac{\partial}{\partial t}$ is given by $$\sec(\sigma)=-1+\frac{4(1-b^2-e^{2a})}{((1+b)e^t+(1-b)e^{-t})^2}\sin^2(\theta).
$$ 

\item For $n\geq 3$, $(M^n,h_{n,a,b})$ Riemannian covers a finite volume manifold if and only if $1=b^2+e^{2a}$, or equivalently, if and only if $(M^n,h_{n,a,b})$ is isometric to $(\mathbb{H}^n,h_n)$.

\end{enumerate}
\end{thm}

\begin{rem}
Property (3) implies (P). 
\end{rem}

\begin{rem}\label{path}
For $n \geq 3$, the Riemannian metric $h_{n,a,0}$ has extremal sectional curvatures $-1$ and $-e^{2a}$ by (4);  consequently, manifolds satisfying (P) can have curvatures pinched arbitrarily close to $-1$ from above or below without being isometric to hyperbolic space. 
\end{rem}

The Main Construction is motivated by the study of \textit{rank rigidity} \cite{MR819559}-\cite{Bettiol:2014aa}, \cite{MR908215}-\cite{MR1860505},  \cite{MR3493370}-\cite{MR3054631}, and in particular in manifolds of \textit{positive hyperbolic rank} \cite{MR1934699, MR2449143, MR1114460, Lin:2016aa}.  A complete Riemannian manifold $M$ has \textit{positive hyperbolic rank} if along each complete geodesic $\gamma:\mathbb{R} \rightarrow M$, there exists a parallel vector field $V(t)$ such that $\sec(\dot{\gamma},V)(t) \equiv -1$.  This hypothesis on geodesics is an infinitesimal analogue of (P).

The negatively curved locally symmetric spaces, normalized so that $-1$ is an extremal value of the sectional curvatures (i.e. an achieved lower or upper bound), have positive hyperbolic rank.  While there are additional examples of manifolds with positive hyperbolic rank (such as in the Main Construction), hyperbolic rank-rigidity results assert that under additional mild hypotheses, the negatively curved locally symmetric spaces constitute all examples.

The known examples of complete and \textit{finite volume} manifolds of higher hyperbolic rank are all locally symmetric.  There are no additional complete and finite volume examples amongst manifolds having $\sec \leq -1$ \cite{MR1114460}.  The same is known to be true amongst the complete and finite volume non-positively curved Euclidean rank one manifolds with suitably pinched curvatures \cite{MR2449143}.  Finally, complete and finite volume three-dimensional manifolds of higher hyperbolic rank are real hyperbolic without any a priori sectional curvature bounds \cite{Lin:2016aa}.   

In contrast, it is known that complete \textit{infinite volume} manifolds of higher hyperbolic rank need not be locally symmetric .  The negatively curved symmetric spaces are characterized amongst the homogeneous manifolds with positive hyperbolic rank and $\sec \leq -1$ in  \cite{MR1934699}. Therein, a non-symmetric example is constructed.  To our knowledge, the examples provided by the Main Construction are the first that are \textit{not} locally homogenous.

\section{Warping hyperbolic $1$-space over the Euclidean line.}

Let $t$ denote the Euclidean coordinate on $\mathbb{R}$ and $r>0$ the Euclidean coordinate on $\mathbb{H}^1$.  

Consider the foliation of $\mathbb{H}^2$ by the family of (Euclidean) upper-half semicircles with common center the origin in $\mathbb{R}^2$.  Each such semicircle, parameterized appropriately, is an $h_2$-geodesic: For each $r>0$, the map $$F_r:\mathbb{R} \rightarrow \mathbb{H}^2$$ defined by $F_r(t)= (r\tanh(t),r\sech(t))$ parameterizes the semicircle through $(0,r)$ in the clockwise fashion as a unit-speed $h_2$-geodesic with the initial point $(0,r)$.

Define a diffeomorphism $$F:\mathbb{R} \times \mathbb{H}^1 \rightarrow \mathbb{H}^2$$ by $F(t,r)=F_r(t)$.  In the coordinates $(t,r)$ on $\mathbb{R} \times \mathbb{H}^1$, the metric $F^{*}(h_2)$ is given by $$F^{*}(h_2)=dt^2+\cosh^2(t)h_1=dt^2+e^{2 \ln(\cosh(t))}h_1.$$ Hence, $(\mathbb{H}^2,h_2)$ is isometric to a warped product of the hyperbolic line over the Euclidean line.  

In fact, the hyperbolic plane $\mathbb{H}^2$ is isometric to many different warpings of $\mathbb{H}^1$ over $\mathbb{R}$.  The warping functions are described in terms of solutions to the second order differential equation 

\begin{equation}\label{jacobi}
\phi''(t)+(\phi'(t))^2-1=0.
\end{equation} 

For $a \in \mathbb{R}$ and $-1\leq b \leq 1$, the solution to (\ref{jacobi}) determined by the initial conditions $\phi(0)=-a$ and $\phi'(0)=b$ is given by 

\begin{equation}\label{soln}
\phi_{a,b}(t)=\ln(\frac{(1+b)e^t+(1-b)e^{-t}}{2e^{a}}).
\end{equation}

Let $M^2=\mathbb{R} \times \mathbb{H}^1$ and let $\pi:M^2 \rightarrow \mathbb{R}$ be the first coordinate projection.  Ignore the slight abuse of notation in letting $\phi_{a,b}$ also denote the function $\pi^{*}(\phi_{a,b})$ on $M^2$.

For $a$ and $b$ as above, consider the warped product Riemannian metric $$h_{2,a,b}=dt^2+e^{2\phi_{a,b}}h_1$$ on $M^2$.  In this notation, $h_{2,0,0}=F^{*}(h_2).$ 

\begin{lem}\label{nonew}
For each $(a,b) \in \mathbb{R} \times [-1,1]$, $(M^2,h_{2,a,b})$ is isometric to the hyperbolic plane $(\mathbb{H}^2,h_2)$.  
\end{lem}

\begin{proof}
The proof only uses the fact that $\phi_{a,b}$ is a solution to (\ref{jacobi}).  For this reason, and to simplify notation, set $h=h_{2,a,b}$ and $\phi=\phi_{a,b}$ in the remainder of this proof.   As $(\mathbb{R},dt^2)$ and $(\mathbb{H}^1,h_1)$ are complete, so too is their warped product $(M^2, h)$ \cite[Lemma 40, pg. 209]{MR719023}.  Therefore, it suffices to prove that $h$ has constant curvature $-1$.

The curvature at each point equals $-1$ as a consequence of (\ref{jacobi}) since by \cite[Proposition 42 (4), pg. 210]{MR719023}, the sectional curvature at a point is given by $$-\phi''-(\phi')^2.$$

%vector fields $X$,  $Y$, anDefine d $J$ on $\mathbb{R} \times \mathbb{H}^1$ by $$X(t,r)=\frac{\partial}{\partial t}(t,r),$$  $$Y(t,r)=re^{-\phi(t)}\frac{\partial}{\partial r}(t,r),$$ and $$J(t,r)=e^{\phi(t)}Y(t,r).$$  As  $\pi:(M^2,h) \rightarrow (\mathbb{R},dt^2)$ is a Riemannian submersion, the integral curves of $X$ are unit-speed geodesics.  As $X$ and $Y$ are $h$-orthonomal, $Y$ restricts to a parallel field along each such geodesic.  As $[X,J]=0$, $J$ restricts to a Jacobi field along each such geodesic.

%It follows that for each $(t,r)\in M^2,$ $$\phi''(t)+(\phi'(t))^2+\sec(t,r)=0,$$ which taken with (\ref{jacobi}) implies that $\sec(t,r)=-1$. 
\end{proof}

\section{Main Construction: Warping hyperbolic $(n-1)$-space over the Euclidean line.}

Fix $n \geq 2$ and  $(a,b) \in \mathbb{R} \times [-1,1]$. Let $\phi_{a,b}$ be as in $(\ref{soln})$.  Let $h_{n,a,b}$ denote the warped product Riemannian metric on $M^n=\mathbb{R} \times \mathbb{H}^{n-1}$ defined by  $$h_{n,a,b}=dt^2+e^{2\phi_{a,b}}h_{n-1}.$$  The metric $h_{n,a,b}$ is complete by \cite[Lemma 40, pg. 209]{MR719023}  We now consider the properties (1)-(5) stated in the Main Construction.\\

\noindent \textbf{Property (1):} Property (1) is immediate since the warping functions $\phi_{a,b}$ depend smoothly on $(a,b) \in \mathbb{R} \times [-1,1]$.\\

\noindent \textbf{Property (2):} Given $F \in \Isom(\mathbb{H}^{n-1},h_{n-1})$, define $\bar{F}:M^n \rightarrow M^n$ by $\bar{F}(t,(x_i))=(t,F(x_i)).$  Then $\bar{F} \in \Isom(M^n,h_{n,a,b})$ and the map $$\Isom(\mathbb{H}^{n-1},h_{n-1}) \rightarrow \Isom(M^n,h_{n,a,b})$$ defined by $F \mapsto \bar{F}$ is a monomorphism.\\

\noindent \textbf{Property (3):} To find one such totally-geodesic hyperbolic plane, define $R \in \Isom(\mathbb{H}^{n-1},h_{n-1})$ by $$R(x_1,\ldots,x_{n-2}, x_{n-1})=(-x_1,\ldots, -x_{n-2}, x_{n-1}).$$  As $\bar{R}$ is an isometry, its fixed point set $$\Sigma=\{(t,(0,\ldots,0,x_{n-1}))\}\subset M^n$$ is a totally-geodesic surface.  The map $(t,(0,\ldots,0,x_{n-1})) \mapsto (t,x_{n-1})$ defines an isometry between the induced metric on $\Sigma$ and $(M^2,h_{2,a,b})$.  By Lemma \ref{nonew}, $\Sigma$ is isometric to the hyperbolic plane.  We conclude by showing that given an arbitrary tangent plane $\sigma$ as in Property (3) there is an isometry of $M^n$ that carries a tangent plane to $\Sigma$ to $\sigma$.

Let $\pi:M^n \rightarrow \mathbb{R}$ denote the first coordinate projection and $F_{\bar{t}}$ the fiber above $\bar{t}\in \mathbb{R}$.  For each $p \in F_{\bar{t}}$, let $G_p$ denote the set of tangent $2$-planes to $M^n$ at $p$ that contain the vector $\frac{\partial}{\partial t}(p)$.  Let $$X_{\bar{t}}=\cup_{p \in F_{\bar{t}}} G_p.$$   Let $\sigma_0$ denote the tangent space to $\Sigma$ at the point $(\bar{t},(0,\ldots,0,1))$ and note that $\sigma_0 \in X_{\bar{t}}$.  As $\Isom(\mathbb{H}^{n-1},h_{n-1})$ acts transitively on unit-tangent vectors to $\mathbb{H}^{n-1}$, $\Isom(M^n,h_{n,a,b})$ acts transitively on the set $X_{\bar{t}}$.  Hence, there exists an isometry $\bar{I}$ of $M^n$ that carries $\sigma_0$ to $\sigma$.  

As $$\exp(\sigma)=\exp(d\bar{I}(\sigma_0))=\bar{I}(\exp(\sigma_0))=\bar{I}(\Sigma),$$ $\exp(\sigma)$ is a totally geodesic surface isometric to the hyperbolic plane.\\

\noindent \textbf{Property (4):}  Let $p=(t,(x_i)) \in M$ and let $v,w,\frac{\partial}{\partial t}$ be orthonormal vectors at $p$.

Property (3) implies that 
\begin{equation}\label{tang}
R(\frac{\partial}{\partial t},w,w,\frac{\partial}{\partial t})=-1.
\end{equation}

By  \cite[Proposition 42 (5), pg. 210]{MR719023},
\begin{equation}\label{perp}
R(v,w,w,v)=-1+\frac{4(1-b^2-e^{2a})}{((1+b)e^t+(1-b)e^{-t})^2}.
\end{equation}

By \cite[Proposition 42 (3), pg. 210]{MR719023},
\begin{equation}\label{mix}
R(\frac{\partial}{\partial t},w)v=0.
\end{equation}

%$\sigma \subset T_pM$ be a two-dimensional subspace.  Then, by Property (3),

%\begin{equation}\label{tang}
%\frac{\partial}{\partial t} \in \sigma \implies \sec(\sigma)=-1,
%\end{equation}

%and by \cite[Proposition 42 (5)]{MR719023},

%\begin{equation}\label{perp}
%\frac{\partial}{\partial t} \perp \sigma \implies \sec(\sigma)=-1+\frac{4(1-b^2-e^{2a})}{((1+b)e^t+(1-b)e^{-t})^2}.
%\end{equation}

%\begin{lem}\label{perp}
%For $p=(t,(x_i)) \in M$ and $\sigma \subset T_pM$ a two-dimensional subspace
%\begin{enumerate}
%\item If $\frac{\partial}{\partial t} \in \sigma$, then $\sec(\sigma)=-1$.
%\item If $\frac{\partial}{\partial t} \perp \sigma$, then $$\sec(\sigma)=-1+\frac{4(1-b^2-e^{2a})}{((1+b)e^t+(1-b)e^{-t})^2}.$$
%\end{enumerate}
%\end{lem}

%\begin{proof}
%Part (1) in the Lemma is immediate from Property (3) above.  If $\sigma$ is a $2$-plane as in Part (2) of the Lemma, then by \cite[Proposition 42 (5)]{MR719023},
%  $$\sec(\sigma)=-(\phi_{n,a,b}'(t))^2-e^{-2\phi_{n,a,b}(t)}$$ which simplifies to the formula in Part (2) of the lemma.
%\end{proof}

%\begin{lem}\label{mixed}
%If $v$, $w$, and $\frac{\partial}{\partial t}$ are orthonormal, then $R(v,w)\frac{\partial}{\partial t}=0$.
%\end{lem}

%\begin{proof}
%By Lemma \ref{shape} and the fact that $\phi'_{n,a,b}$ is constant on fibers of $\pi$, 

%\begin{eqnarray*}
%R(v,w)\frac{\partial}{\partial t} =
%-\nabla_{v} S(w)+\nabla_{w} S(v)+S([v,w])\\
%=-\nabla_{v} \phi'_{n,a,b} w + \nabla_{w} \phi'_{n,a,b} v+\phi'_{n,a,b} [v,w]\\
%=0.
%\end{eqnarray*}
%\end{proof}

Now fix a two dimensional subspace $\sigma \subset T_pM$.  Assume that $\sigma$ makes angle $\theta \in [0, \frac{\pi}{2}]$ with $\frac{\partial}{\partial t}$.  Then there exist unit length vectors $v$ and $w$ perpendicular to $\frac{\partial}{\partial t}$ such that   $$\{\bar{v}=\cos(\theta)\frac{\partial}{\partial t}+\sin(\theta)v,w\} $$ is an orthonormal basis of $\sigma$.  By (\ref{tang})-(\ref{mix}) and the symmetries of the curvature tensor, 
\begin{equation*}
\begin{split}
\sec(\sigma)&=R(\bar{v},w,w,\bar{v})\\
&=\cos^2(\theta)R(\frac{\partial}{\partial t},w,w,\frac{\partial}{\partial t})+\sin^2(\theta)R(v,w,w,v)\\
&=-1+\frac{4(1-b^2-e^{2a})}{((1+b)e^t+(1-b)e^{-t})^2}\sin^2(\theta).
\end{split}
\end{equation*}

\noindent \textbf{Property (5):} As $(M^n,h_{n,a,b})$ is simply connected and complete, it is isometric to $(\mathbb{H}^n,h_n)$ if and only if it has constant sectional curvatures $-1$.  By (4), this is equivalent to the parameters satisfying $1=b^2+e^{2a}.$ As $(\mathbb{H}^n,h_n)$ Riemannian covers finite volume manifolds, it remains to prove that when $n \geq 3$ and $1\neq b^2+e^{2a}$, $(M^n,h_{n,a,b})$ does not Riemannian cover a finite volume manifold.

Assume that $1 \neq b^2+e^{2a}$.  As isometries preserve sectional curvatures, Property (4) implies that each isometry $F \in \Isom(M^n,h_{n,a,b})$ has the form $F((t,x_i))=(F_1(t), F_2(x_i))$ for some $F_1 \in \Isom(\mathbb{R},dt^2)$ and $F_2\in \Isom(\mathbb{H}^{n-1},h_{n-1})$.  If $b=\pm1$, then $F_1$ is the identity.  If $b \neq \pm 1$, then $F_1$ is either the identity or reflection about the critical point $t_0=\ln(\sqrt{1-b})-\ln(\sqrt{1+b})$ of $((1+b)e^t+(1-b)e^{-t})^2$.

Let $(N,g)$ be a Riemannian manifold with universal covering $(M^n,h_{n,a,b})$.  Its fundamental group $\pi_1(N)$ is identified with a subgroup $\Gamma$ of $\Isom(M,h_{n,a,b})$ that acts properly discontinuously and fixed point freely on $M^n$.  The association $\gamma \mapsto \gamma_1$ described above defines a homomorphism $\Phi: \Gamma \rightarrow \Isom(\mathbb{R},dt^2)$ whose kernel has index at most two in $\Gamma$.  Let $(\bar{N},\bar{g})$ be the Riemannian covering of $(N,g)$ associated to the subgroup $\ker \Phi$ of $\Gamma$.   

As each isometry in $\ker \Phi$ acts trivially on the first factor of $M^n=\mathbb{R} \times \mathbb{H}^{n-1}$,  the metric $\bar{g}$ is also a warped product metric with the same warping function $e^{\phi_{n,a,b}}$.  By (4), $\bar{g}$ has negative sectional curvatures.  By \cite[Lemma 7.6]{MR0251664}, $e^{\phi_{n,a,b}}$ is a non-constant convex function on $(\bar{N},\bar{g})$.  By \cite[Proposition 2.2]{MR0251664}, $(\bar{N},\bar{g})$ has infinite volume.  Therefore, $(N,g)$ has infinite volume, as required.

\section{Concluding remarks and questions.}

\begin{rem}\label{euc}
For $n \geq 1$, let $g_n$ denote the Euclidean metric on $\mathbb{R}^n$.  An analogous construction replacing $(\mathbb{H}^{n-1},h_{n-1})$ with $(\mathbb{R}^{n-1},g_{n-1})$ produces complete nonpositively curved metrics on $\mathbb{R}^n$ satisfying (P).  In this setting, the parameters $(a,b)=(0,1)$  correspond to the well-known warped product expression for $\mathbb{H}^n$ arising from normal coordinates about the horosphere through $(0,\ldots,0,1) \in \mathbb{H}^{n-1}$. 
\end{rem}

\begin{ques}
Are all complete Riemannian metrics on $M^3$ that satisfy (P) isometric to a metric as described in the Main Construction or in Remark \ref{euc}?
\end{ques}

\begin{ques}
The constant curvature metric belongs to a one parameter family of pairwise non-isometric metrics on the manifold $M^n$ each of which satisfies (P) as in Remark \ref{path}.  Are the symmetric metrics on the other negatively curved symmetric spaces similarly deformable through metrics satisfying (P)?
\end{ques}

In view of the existing hyperbolic rank-rigidity results for finite volume manifolds, it seems reasonable to make the following conjecture: 

\begin{conj}
A finite volume Riemannian manifold that satisfies (P) is locally symmetric.
\end{conj}

%any isometry $I$ of $(M^n, h_{n,a,b})$ preserves the product structure.  In particular, $I^{*}(\phi_{n,a,b})=\phi_{n,a,b}$.  As $\phi_{n,a,b}$ is non-constant, it suffices to prove that $\phi_{n,a,b}$ has positive semidefinite hessian  by \cite[Proposition 2.2]{MR0251664}.

%This is easy to check using Lemma \ref{shape} with respect to the orthonormal framing introduced $\{\frac{\partial}{\partial t},V_1, \ldots, V_{n-1}\}$ introduced in Lemma \ref{shape}.  Routine calculations show that at $p=(t,(x_i))$, 

%$$\Hess_{\phi}(\frac{\partial}{\partial t},\frac{\partial}{\partial t})=\phi''_{n,a,b}(t)=\frac{4(1-b^2)}{((1+b)e^t+(1-b)e^{-t})^2}\geq 0,$$

%$$\Hess_{\phi_{n,a,b}}(V_i,V_j)=\delta_{i}^{j}(\phi'_{n,a,b}(t))^2 \geq 0,$$ and 
%$$\Hess_{\phi_{n,a,b}}(V_i,\frac{\partial}{\partial t})=0,$$ concluding the proof of this property.

\bibliographystyle{plain}
\bibliography{Rank}

\end{document}